\documentclass[11pt]{amsart}
\usepackage[babel]{csquotes}
\usepackage[inline]{enumitem}
\usepackage{amsmath,amsthm,amssymb,mathrsfs,amsfonts,verbatim,enumitem,color,leftidx,mathtools}
\usepackage{mathabx}
\usepackage{etoolbox} 
\usepackage{bbm}
\usepackage[all,tips]{xy}
\usepackage{graphicx,ifpdf}
\usepackage{stmaryrd}
\usepackage{amssymb}
\usepackage{dsfont}

\ifpdf
\fi
\usepackage[colorlinks]{hyperref}
\hypersetup{
linkcolor=blue,          
citecolor=green,        
}

\usepackage{tikz}
\usetikzlibrary{calc}
\usetikzlibrary{shapes.misc}

\tikzset{cross/.style={cross out, draw=black, minimum size=2*(#1-\pgflinewidth), inner sep=0pt, outer sep=0pt},
cross/.default={1pt}}

\newtheorem{thm}{Theorem}[section]
\newtheorem{lem}[thm]{Lemma}
\newtheorem{cor}[thm]{Corollary}
\newtheorem{prop}[thm]{Proposition}

\theoremstyle{definition}
\newtheorem{defi}[thm]{Definition}

\theoremstyle{remark}
\newtheorem{rem}[thm]{Remark}

\numberwithin{equation}{section}
\definecolor{esperance}{rgb}{0.0,0.5,0.0}

\newcommand{\de}{\delta}

\newcommand{\rk}{\mathrm{rk\,}}

\DeclareMathOperator{\diam}{diam}

\newcommand{\Bad}{\mb{Bad}}

\newcommand{\al}{\alpha}
\newcommand{\ga}{\gamma}

\newcommand{\del}{\delta}
\newcommand{\Del}{\Delta}
\newcommand{\lam}{\lambda}

\newcommand{\eps}{\epsilon}

\newcommand{\ka}{\kappa}

\newcommand{\cH}{\mathcal{H}}

\newcommand{\cR}{\mathcal{R}}

\newcommand{\bR}{\mathbb{R}}
\newcommand{\bZ}{\mathbb{Z}}
\newcommand{\bQ}{\mathbb{Q}}

\newcommand{\bT}{\mathbb{T}}

\newcommand\wh[1]{\widehat{#1}}

\newcommand\mb[1]{\mathbf{#1}}

\newcommand\tb[1]{\textbf{#1}}

\newcommand{\onto}{\xymatrix{\ar@{>>}[r]&}}

\newcommand{\eq}[1]
{
\begin{equation*}
{#1}
\end{equation*}
}
\newcommand{\eqlabel}[2]
{
\begin{equation}
{#2}\label{#1}
\end{equation}
}

\makeatletter
\newcommand*{\rom}[1]{\expandafter\@slowromancap\romannumeral #1@}
\makeatother

\begin{document}

\title{On a Kurzweil type theorem via ubiquity}

\date{}

\author{Taehyeong Kim}
\address{The Einstein Institute of Mathematics, Edmond J. Safra Campus, Givat Ram\\
The Hebrew University of Jerusalem, Jerusalem, 91904, Israel}
\email{taehyeong.kim@mail.huji.ac.il}

\thanks{}


\keywords{}

\def\thefootnote{}
\footnote{2020 {\it Mathematics
Subject Classification}: Primary 11J20 ; Secondary 11J83, 28A78}   
\def\thefootnote{\arabic{footnote}}
\setcounter{footnote}{0}

\begin{abstract}
Kurzweil's theorem ('55) is concerned with zero-one laws for well approximable targets in inhomogeneous Diophantine approximation under the badly approximable assumption. In this article, we prove the divergent part of a Kurzweil type theorem via a suitable construction of ubiquitous systems when the badly approximable assumption is relaxed. Moreover, we also discuss some counterparts of Kurzweil's theorem.
\end{abstract}
\maketitle
\section{Introduction}
Kurzweil's theorem \cite{Kur55} in inhomogeneous Diophantine approximation is concerned with well approximable target vectors. We start by introducing related definitions and notations.
Given a decreasing function $\psi:\bR^+ \to \bR^+$ and an $m\times n$ matrix $A\in M_{m,n}(\bR)$, we say that $\mb{b}\in\bR^m$ is \textit{$\psi$-approximable for $A$} if there exist infinitely many solutions $\mb{q}\in\bZ^n$ to the following inequality
\eq{
\|A\mb{q}-\mb{b}\|_{\bZ} < \psi(\|\mb{q}\|).
}
Denote by $W_A(\psi)$ the set of such vectors in the unit cube $[0,1]^m$.
Here and hereafter, $\|\mb{x}\|=\max_{1\leq i\leq m}|x_i|$ and $\|\mb{x}\|_\bZ=\min_{\mb{n}\in\bZ^m} \|\mb{x}-\mb{n}\|$ for $\mb{x}\in\bR^m$. 
We say that $A\in M_{m,n}(\bR)$ is \textit{badly approximable} if 
\[
\liminf_{\|\mb{q}\|\to\infty} \|\mb{q}\|^{\frac{n}{m}}\|A\mb{q}\|_\bZ >0.
\]

Kurzweil proved the following zero-one law for $W_A(\psi)$.
\begin{thm}\cite{Kur55}\label{thm_Kur} If $A\in M_{m,n}(\bR)$ is badly approximable, then for any decreasing $\psi:\bR^+ \to \bR^+$ we have
\eq{
|W_{A}(\psi)|=
  \begin{dcases}
    0   & \quad \text{if } \sum_{q=1}^{\infty} q^{n-1}\psi(q)^m<\infty,\\
    1  & \quad \text{if } \sum_{q=1}^{\infty} q^{n-1}\psi(q)^m=\infty.
  \end{dcases} 
}
Here and hereafter, $|\cdot|$ stands for Lebesgue measure on $\bR^m$.
\end{thm}
We remark that Kurzweil showed that in fact the badly approximable condition is an equivalent condition for the zero-one law, not a sufficient condition.

In this article, we will consider similar results when the badly approximable condition is relaxed.
We say that $A\in M_{m,n}(\bR)$ is \textit{singular} if for any $\eps>0$ for all large enough $X\geq 1$ there exists $\mb{q}\in\bZ^n$ such that
\[
\|A\mb{q}\|_\bZ <\eps X^{-\frac{n}{m}}\quad\text{and}\quad 0<\|\mb{q}\|<X.
\]
Otherwise we call it \textit{non-singular} (or \textit{regular} following \cite{Cas57}).
One can check that $A\in M_{m,n}(\bR)$ is singular if and only if for any $\eps>0$ for all large enough $\ell \in\bZ_{\geq 1}$
there exists $\mb{q}\in\bZ^n$ such that
\eqlabel{eq_sing}{
\|A\mb{q}\|_\bZ <\eps 2^{-\frac{n}{m}\ell}\quad\text{and}\quad 0<\|\mb{q}\|<2^{\ell}.
}
Hence $A\in M_{m,n}(\bR)$ is non-singular if and only if there exists $\eps >0$ such that the set
\[
L(\eps):=\{\ell\in\bZ_{\geq 1}: \text{ there is no solution }\mb{q}\in\bZ^n\text{ to }\eqref{eq_sing}\text{ with }\ell\}
\]
is unbounded.
We call $L(\eps)$ \textit{$\eps$-return sequence for $A$}. 

\begin{rem}\label{rem_bad}\
\begin{enumerate}
\item\label{item_bad} Note that $A\in M_{m,n}(\bR)$ is badly approximable if and only if there exists $\eps>0$ such that $L(\eps) = \bZ_{\geq 1}$.
\item In a dynamical point of view as in \cite{D85}, the set $L(\eps)$ corresponds to return times to a compact set related to $\eps$ of a certain diagonal flow in the space of lattices.
\end{enumerate}
\end{rem}

The following is the main theorem of this article. 
\begin{thm}\label{thm_main} 
Let $A\in M_{m,n}(\bR)$ be non-singular with $\eps$-return sequence $L(\eps)=\{\ell_i\}_{i\geq 1}$. For any decreasing $\psi:\bR^+ \to \bR^+$ and $0\leq s\leq m$, the $s$-dimensional Hausdorff measure of $W_{A}(\psi)$ is given by
\[
\cH^s(W_{A}(\psi))=\cH^s([0,1]^m)\quad\text{if}\quad\sum_{i=1}^{\infty} 2^{\ell_i n}\psi(2^{\ell_i})^s=\infty
\]
\end{thm}

For $\del>0$, let $\psi_\del (q) = \del q^{-\frac{n}{m}}$. Denote $\Bad_A(\del) = [0,1]^m\setminus W_A(\psi_\del)$ and $\Bad_A = \bigcup_{\del>0}\Bad_A(\del)$. 
Theorem \ref{thm_main} with $\psi=\psi_\del$ and $s=m$ directly implies the following corollary.
\begin{cor}\label{cor_bad}
If $A\in M_{m,n}(\bR)$ is non-singular, then for any $\del>0$, the set $\Bad_A(\del)$ has Lebesgue measure zero, hence, $\Bad_A$ has Lebesgue measure zero.
\end{cor}
\begin{rem}\
\begin{enumerate}
\item There are some historical remarks about Corollary \ref{cor_bad}. The one-dimensional result of the corollary was proved in \cite{Kim07} using irrational rotations and the Ostrowski representation. The corollary in full generality was proved in \cite{Sha13} using a certain mixing property in homogeneous dynamics. Simultaneous version (i.e. $n=1$) of the corollary was proved in \cite{Mos} using a certain well distributed property. Our method relies on a suitable construction of a ubiquitous system. 
\item A zero-one law for Lebesgue measure of $W_A(\psi)$ in one-dimensional case was investigated in \cite{FK16}. According to their results, Theorem \ref{thm_main} is not optimal. It seems very interesting to obtain zero-one laws for $W_A(\psi)$ in multidimensional case.
\item A weighted version of Kurzweil's theorem was investigated in \cite{Har12}. There have been several recent results on the weighted ubiquity and weighted transference theorems (see \cite{CGGMS20}, \cite{G20}, and \cite{WW21}). It seems plausible to utilize these results to obtain a weighted version of Theorem \ref{thm_main}.
\end{enumerate}
\end{rem}

As stated in \cite[Section 9]{BBDV09}, using Theorem \ref{thm_Kur} and Mass Transference Principle in \cite{BV06}, we are able to deduce Hausdorff measure version of Kurzweil's theorem. Theorem \ref{thm_main} which relies on a ubiquity method also implies the following corollary.
\begin{cor}\label{cor_zero-one} If $A\in M_{m,n}(\bR)$ is badly approximable, then for any decreasing $\psi:\bR^+ \to \bR^+$ and $0\leq s\leq m$, we have
\eq{
\cH^s(W_{A}(\psi))=
  \begin{dcases}
    0   & \quad \text{if } \sum_{q=1}^{\infty} q^{n-1}\psi(q)^s<\infty,\\
    \cH^s([0,1]^m)  & \quad \text{if } \sum_{q=1}^{\infty} q^{n-1}\psi(q)^s=\infty.
  \end{dcases} 
} Moreover, the convergent part holds for every $A\in M_{m,n}(\bR)$.
\end{cor}
\begin{proof}
The convergent part will be proved in Section \ref{sec2}.
Since the divergence and convergence of the sums
\eq{
\sum_{\ell=1}^{\infty} 2^{\ell n}\psi(2^{\ell})^s\quad\text{and}\quad \sum_{q=1}^{\infty} q^{n-1}\psi(q)^s\quad \text{coincide},
} the divergent part follows from Theorem \ref{thm_main} and Remark \ref{rem_bad} \eqref{item_bad}.
\end{proof}

We explore some counterparts of Kurzweil's theory. 
We denote by $w(A,\mb{b})$ the supremum of the real numbers $w$ for which, for arbitrarily large $X$, the inequalities
$$\|A\mb{q}-\mb{b}\|_\bZ < X^{-w} \quad\text{and}\quad\|\mb{q}\| < X$$
have an integral solution $\mb{q}\in\bZ^n$. We also denote by $\wh{w}(A)$ the supremum of the real numbers $w$ for which, for all sufficiently large $X$, the inequalities
$$\|A\mb{q}\|_\bZ < X^{-w} \quad\text{and}\quad\|\mb{q}\| < X$$
have an non-zero integral solution $\mb{q}\in\bZ^n$. 
If $\wh{w}({^tA})>\frac{m}{n}$, then by \cite[Theorem]{BL05}, for almost all $\mb{b}\in\bR^m$,
\[
w(A,\mb{b}) = \frac{1}{\wh{w}({^tA})} < \frac{n}{m}.
\] Thus for any $\del>0$, the Lebesgue measure of $\Bad_A(\del)$ is full. This is opposite to Corollary \ref{cor_bad}.
Note that if $\wh{w}({^tA})>\frac{m}{n}$, then $^{t}A$ is singular, hence $A$ is singular. So, they do not conflict with each other. 

If we consider the case $\wh{w}({^tA})=\frac{m}{n}$ and $^tA$ is singular, then we cannot deduce from \cite[Theorem]{BL05} that $\Bad_A(\del)$ is of full Lebesgue measure for any $\del>0$. We will give a certain sufficient condition for $\Bad_A(\del)$ being of full Lebesgue measure for any $\del>0$. 

If $\rk_\bZ ({^{t}A}\bZ^m + \bZ^n) < m+n$, then $\wh{w}({^tA})=\infty$, hence we may assume that $\rk_\bZ ({^{t}A}\bZ^m + \bZ^n) = m+n$. Then following \cite[Section 3]{BL05}, there exists a sequence of best approximations $(\mb{y}_k)_{k\geq 1}$ in $\bZ^m$ for $^t A$. Denote $Y_k = \|\mb{y}_k\|$ and $M_k = \|{^t A}\mb{y}_k\|_\bZ$.

\begin{thm}\label{thm_counter}
If 
\[
\sum_{k\geq 2}\max\left(\left(Y_{k}^{\frac{m}{n}}M_{k-1}\right)^{\frac{n}{m+n}}, \left(Y_{k+1}^{\frac{m}{n}}M_{k}\right)^{\frac{n}{m+n}}\right) < \infty,
\]
then for any $\del>0$, the Lebesgue measure of $\Bad_A(\del)$ is full.
\end{thm} 
\begin{rem}\
\begin{enumerate}
\item The summability assumption implies that $Y_{k+1}^{\frac{m}{n}}M_{k} \to 0$ as $k\to \infty$, hence $^tA$ is singular.
\item This theorem is stronger than the previous observation because if $\wh{w}({^t A}) > \frac{m}{n}$, then there is $\gamma>0$ such that $Y_{k+1}^{\frac{m}{n}+\gamma}M_k < 1$ for all sufficiently large $k\geq 1$. Hence,
\[
\sum_{k\geq 2}\max\left(\left(Y_{k}^{\frac{m}{n}}M_{k-1}\right)^{\frac{n}{m+n}}, \left(Y_{k+1}^{\frac{m}{n}}M_{k}\right)^{\frac{n}{m+n}}\right) < \sum_{k\geq 1} \left(Y_{k+1}^{-\gamma}\right)^{\frac{n}{m+n}} < \infty
\] since $Y_k$ increases at least geometrically (see \cite[Lemma 1]{BL05}).
\item It was proved in \cite{BKLR, KKL} that $A$ is \textit{singular on average} if and only if there exists $\del>0$ such that 
$\Bad_A(\del)$ has full Hausdorff dimension. Thus it seems very interesting to obtain an equivalent Diophantine property of $A$ for $\Bad_A(\del)$ being of full Lebesgue measure.
\end{enumerate}
\end{rem}
 
The structure of this paper is as follows: 
In Section \ref{sec2}, we prove the convergent part of Corollary \ref{cor_zero-one}. In Section \ref{sec3}, we introduce some preliminaries for the proof of Theorem \ref{thm_main} including ubiquitous systems, Transference Principle, and Weyl type uniformly distribution. We prove Theorem \ref{thm_main} and Theorem \ref{thm_counter} in Section \ref{sec4} and Section \ref{sec5}, respectively. 
 
\vspace{5mm}
\tb{Acknowledgments}. 
I would like to thank Victor Beresnevich and Nikolay Moshchevitin for providing helpful comments.

\section{Convergent part: a warm up}\label{sec2}
In this section, we prove the convergent part of Corollary \ref{cor_zero-one}. We will use the following Hausdorff measure version of the Borel-Cantelli lemma \cite[Lemma 3.10]{BD99}. 
\begin{lem}[Hausdorff-Cantelli]\label{hcan}
Let $\{B_i\}_{i\geq1}$ be a sequence of subsets in $\bR^m$. For any $0\leq s\leq k$,
\eq{
\cH^{s}(\limsup_{i\to\infty}B_i)=0\quad\text{if}\quad\sum_{i}\textup{diam}(B_i)^s< {\infty}.
}
\end{lem}

Note that
\[
W_A(\psi)=\limsup_{\|\mb{q}\|\to\infty} B(A\mb{q}, \psi(\|\mb{q}\|)),
\]
where $B(A\mb{q}, \psi(\|\mb{q})\|)$ denotes the ball in $\bR^m$ of radius $\psi(\|\mb{q}\|)$ centered at $A\mb{q}$ modulo $1$,
and $$\sum_{\mb{q}\in\bZ^n} \diam(B(A\mb{q}, \psi(\|\mb{q}\|)))^s < \infty \iff \sum_{q=1}^{\infty} q^{n-1}\psi(q)^s<\infty.$$
Hence, using Hausdorff-Cantelli lemma, we have that for any $0\leq s\leq m$,
\eq{
\cH^s(W_A(\psi)) = 0 \quad\text{if}\quad \sum_{q=1}^{\infty} q^{n-1}\psi(q)^s<\infty.
} This proves the convergent part of Corollary \ref{cor_zero-one}.

\section{Preliminaries for divergent part}\label{sec3}
\subsection{Ubiquity systems}\label{subsec3.1}
The proof of Theorem \ref{thm_main} is based on the ubiquity framework developed in \cite{BDV06}, which provides a very general and abstract approach for establishing the Hausdorff measure of a large class of limsup sets. 
In this subsection, we set up ubiquitous systems that suits our situation.

We consider $\bT^m$ with the supremum norm $\|\cdot\|$. With notation in \cite{BDV06}
we set up the following: 
\[J= \{\mb{q}\in \bZ^{n} \},\quad R_{\mb{q}}=A\mb{q}\in\bT^m, \quad \cR=\{R_{\mb{q}}:\mb{q}\in J\},\quad\beta_{\mb{q}}= \|\mb{q}\|.
\]
Let $l = \{l_i\}$ and $u=\{u_i\}$ be positive increasing sequences such that
\[
l_i < u_i  \quad\text{and}\quad \lim_{i\to\infty} l_i =\infty.
\]
For a decreasing function $\psi :\bR^{+}\to\bR^{+}$, we define
\[
\Del_{l}^{u}(\psi,i):=\bigcup_{\mb{q}\in\bZ^n : l_i<\|\mb{q}\|\leq u_i} B(A\mb{q},\psi(\|\mb{q}\|)).
\]
It follows that
\[
W_A(\psi)=\limsup_{i\to\infty}\Del_l^u(\psi,i):=\bigcap_{M=1}^\infty \bigcup_{i=M}^\infty \Del_l^u(\psi,i).
\]

Throughout, $\rho : \bR^{+}\to\bR^{+}$ will denote a function satisfying $\lim_{r\to\infty}\rho(r)=0$ and is referred to as the \textit{ubiquitous function}. Let
\eq{
\Del_l^u(\rho,i):= \bigcup_{\mb{q}\in\bZ^n : l_i<\|\mb{q}\|\leq u_i} B(A\mb{q},\rho(u_i)).
}

\begin{defi}[Local ubiquity]
Let $B$ be an arbitrary ball in $\bT^{m}$. Suppose that there exist a ubiquitous function $\rho$ and an absolute constant $\ka>0$ such that
\eqlabel{LoUb}{
|B\cap \Del_l^u(\rho,i)| \geq \ka |B|\quad\text{for } i\geq i_{0}(B).
}
Then the pair $(\cR,\beta)$ is said to be a \textit{locally ubiquitous} system relative to $(\rho,l,u)$. 
\end{defi}

Finally, a function $h$ is said  to be \textit{$u$-regular} if there exists a positive constant $\lam<1$ such that for $i$ sufficiently large 
\eq{
h(u_{i+1})\leq \lam h(u_{i}).
}

With notation in \cite{BDV06}, the Lebesgue measure on $\bT^m$ is of type (M2) with $\de=m$ and the intersection conditions are also satisfied with $\ga=0$. These conditions are not stated here but these extra conditions exist and need to be established for the more abstract ubiquity. 

Combining \cite[Corollary 2 and Corollary 4]{BDV06}, we have the following theorem.
\begin{thm}\cite{BDV06}\label{UbThm1}
Suppose that $(\cR,\beta)$ is a local ubiquitous system relative to $(\rho,l,u)$ and assume further that $\rho$ is $u$-regular.
Then for any $0\leq s\leq m$
\eq{
\cH^{s}(W_A(\psi))=\cH^{s}(\bT^m)\quad \text{if}\quad \sum_{i=1}^{\infty}\frac{\psi(u_i)^{s}}{\rho(u_i)^{m}}=\infty.
}
\end{thm}

\subsection{Transference principle}
We need the following transference principle between homogeneous and inhomogeneous Diophantine approximation. See 
\cite[Chapter \rom{5}, Theorem \rom{6}]{Cas57}).
\begin{thm}[Transference principle]\label{thm_trasf}
Suppose that there is no solution $\mb{q}\in\bZ^n\setminus\{0\}$ such that
\[
\|A\mb{q}\|_\bZ < C \quad\text{and}\quad \|\mb{q}\|<X.
\]
Then for any $\mb{b}\in\bR^m$, there exists $\mb{q}\in \bZ^n$ such that
\[
\|A\mb{q}-\mb{b}\|_\bZ \leq C_1 \quad\text{and}\quad \|\mb{q}\|\leq X_1,
\] where 
\[
C_1=\frac{1}{2}(h+1)C,\quad X_1 =\frac{1}{2}(h+1)X,\quad h=X^{-n}C^{-m}.
\]
\end{thm}
This principle implies the following corollary.
\begin{cor}\label{cor_trasf}
Let $A\in M_{m,n}(\bR)$ be non-singular and let $L(\eps)=\{\ell_i\}_{i\geq 1}$ be the $\eps$-return sequence for $A$. Then for any $\mb{b}\in\bR^m$, there exists $\mb{q}\in\bZ^n$ such that
\[
\|A\mb{q}-\mb{b}\|_\bZ \leq \frac{1}{2}(\eps^{-m}+1)\eps 2^{-\frac{n}{m}\ell_i}\quad\text{and}\quad \|\mb{q}\|\leq \frac{1}{2}(\eps^{-m}+1)2^{\ell_i}.
\]
\end{cor}
\begin{proof}
It follows directly from Theorem \ref{thm_trasf} with $C=\eps 2^{-\frac{n}{m}\ell_i}$ and $X=2^{\ell_i}$.
\end{proof}

\subsection{Weyl type uniform distribution}
In this subsection, we will show Weyl type uniform distribution result for the sequence $\left\{A\mb{q}\right\}_{\mb{q}\in\bZ^n} \subset \bT^m$. For $A\in M_{m,n}(\bR)$, Kronecker's theorem (see e.g. \cite[Chapter \rom{3}, Theorem \rom{4}]{Cas57}) asserts that the sequence $\left\{A\mb{q}\right\}_{\mb{q}\in\bZ^n}$ is dense in $\bT^m$ if and only if the subgroup
\[
G({^{t}A}):= {^{t}A}\bZ^m + \bZ^n \subset \bR^n
\] has maximal rank $m+n$ over $\bZ$.
If $A$ is non-singular, then ${^t}A$ is non-singular, hence $G({^{t}A})$ has maximal rank $m+n$ over $\bZ$. By Kronecker's theorem, the sequence $\left\{A\mb{q}\right\}_{\mb{q}\in\bZ^n}$ is dense in $\bT^m$.

But the dense result is not enough for our purpose. We need the following Weyl type uniform distribution result.
\begin{prop}\label{prop_dist}
If $G({^{t}A})$ has maximal rank $m+n$ over $\bZ$, then the sequence $\left\{A\mb{q}\right\}_{\mb{q}\in\bZ^n}$ is uniformly distributed in the following sense: for any ball $B\subset \bT^m$
\[
\frac{\#\{A\mb{q}\in B : \|\mb{q}\|\leq N\}}{\#\{\mb{q}\in\bZ^n : \|\mb{q}\|\leq N\}} \to |B| \quad\text{as}\quad N\to\infty.
\]
\end{prop}
\begin{proof}
We first claim that for any $\mb{c}\in\bZ^m\setminus\{0\}$
\[
\frac{1}{\#\{\mb{q}\in\bZ^n : \|\mb{q}\|\leq N\}}\sum_{\mb{q}\in\bZ^n : \|\mb{q}\|\leq N}e^{2\pi i (\mb{c}\cdot A\mb{q})}  \to 0 \quad \text{as}\quad N\to \infty.
\] Indeed, by the maximal rank assumption, we have ${^t}A\mb{c} \in \bR^n\setminus\bQ^n$. Without loss of generality, we may assume that the first coordinate, say $\alpha$, of ${^t}A\mb{c}$ is irrational. It follows from $\mb{c}\cdot A\mb{q}={^t}A\mb{c}\cdot\mb{q}$ that
\[\begin{split}
\frac{1}{N^n}\left|\sum_{\mb{q}\in\bZ^n : \|\mb{q}\|\leq N}e^{2\pi i (\mb{c}\cdot A\mb{q})} \right|
&= \frac{1}{N^n}\left|\sum_{\mb{q}\in\bZ^n : \|\mb{q}\|\leq N}e^{2\pi i ({{^t}A}\mb{c}\cdot \mb{q})} \right|\\
&\ll \frac{1}{N^n}N^{n-1}\sum_{q_1=-N}^N e^{2\pi i \al q_1}\\
&\leq \frac{1}{N}\frac{2}{|e^{2\pi i \al}-1|}.
\end{split}\] Since $\#\{\mb{q}\in\bZ^n : \|\mb{q}\|\leq N\}\asymp N^n$, this proves the claim.

Following the proof of classical Weyl's criterion (see e.g. \cite[Theorem 2.1]{KN74}), we can deduce that 
for any ball $B \subset \bT^m$ we have
\[
\frac{\#\{A\mb{q}\in B : \|\mb{q}\|\leq N\}}{\#\{\mb{q}\in\bZ^n : \|\mb{q}\|\leq N\}} \to |B| \quad\text{as}\quad N\to\infty.
\]
\end{proof}
\begin{rem}
The above proposition is slightly different to the multidimensional Weyl's criterion. We do not take every partial sum but ``radial'' partial sum.
\end{rem}

\section{Proof of Theorem \ref{thm_main}}\label{sec4}
Let $A$ be non-singular and let $L(\eps)=\{\ell_i\}_{i\geq 1}$ be the $\eps$-return sequence. 
With the notations in Subsection \ref{subsec3.1}, we take sequences $l=l(\eps)=\{l_i\}$ and $u=u(\eps)=\{u_i\}$ as follows:
\[
u_{i} = \frac{1}{2}(\eps^{-m}+1)2^{\ell_i}\quad\text{and}\quad l_{i}= c_1 u_i,
\]
with some positive constant $c_1=c_1(\eps)<1$, which will be determined later.

We first establish the following local ubiquity with the set-up in Subsection \ref{subsec3.1}.
\begin{thm}\label{LUb}
The pair $(\cR,\beta)$ is a locally ubiquitous system relative to $\left(\rho(r)= c_2 r^{-\frac{n}{m}},l,u\right)$ with the constant
$c_2 = \eps \left(\frac{1}{2}(\eps^{-m}+1)\right)^{1+\frac{n}{m}}$.
\end{thm}
\begin{proof}
Fix any ball $B=B(x,r_0)$ in $\bT^m$. 
By Corollary \ref{cor_trasf}, we have that 
\[
B=B\cap \bigcup_{\mb{q}\in\bZ^n: \|\mb{q}\|\leq u_i} B\left(A\mb{q},\frac{1}{2}(\eps^{-m}+1)\eps 2^{-\frac{n}{m}\ell_i}\right).
\]
Since $$\frac{1}{2}(\eps^{-m}+1)\eps 2^{-\frac{n}{m}\ell_i} = \eps \left(\frac{1}{2}(\eps^{-m}+1)\right)^{1+\frac{n}{m}} u_i^{-\frac{n}{m}}=\rho(u_i),$$ it follows that 
\eqlabel{Eq_Sep}{\begin{split}
|B|\leq &\left|B\cap\bigcup_{\mb{q}\in\bZ^n: \|\mb{q}\|\leq l_i} B(A\mb{q},\rho(u_i))\right|\\
&+\left|B\cap \bigcup_{\mb{q}\in\bZ^n: l_i<\|\mb{q}\|\leq u_i} B(A\mb{q},\rho(u_i))\right|
\end{split}}
Using Proposition \ref{prop_dist} with $2B=B(x,2r_0)$, 
there is an absolute constant $C>0$ independent of the choice of $B$ such that for all large enough $i\geq 1$
\[
\#\{A\mb{q}\in 2B : \|\mb{q}\|\leq l_i \} \leq Cl_i^n |B|. 
\]
Thus for all large enough $i\geq 1$ so that $\rho(u_i)<r_0$, we have
\[
\left|B\cap\bigcup_{\mb{q}\in\bZ^n: \|\mb{q}\|\leq l_i} B(A\mb{q},\rho(u_i))\right| 
\leq Cl_i^n|B| (2\rho(u_i))^m = (2c_2)^m C c_1^n |B|. 
\]
By taking $0<c_1<1$ so that $(2c_2)^m C c_1^n<\frac{1}{2}$, which depends only on $\eps$, 
it follows from \eqref{Eq_Sep} that for all large enough $i\geq 1$
\eq{\begin{split}
|B\cap \Del_l^u(\rho,i)| &=\left|B\cap \bigcup_{\mb{q}\in\bZ^n: l_i<\|\mb{q}\|\leq u_i} B(A\mb{q},\rho(u_i))\right|\\
&\geq |B|- (2c_2)^m C c_1^n |B| >\frac{1}{2}|B|.
\end{split}}
\end{proof}

\begin{proof}[Proof of Theorem \ref{thm_main}]
It follows from $\ell_{i+1} \geq \ell_{i}+1$ that for any $i\geq 1$
\[
\rho(u_{i+1})=\frac{1}{2}(\eps^{-m}+1)\eps 2^{-\frac{n}{m}\ell_{i+1}} \leq 2^{-\frac{n}{m}}\frac{1}{2}(\eps^{-m}+1)\eps 2^{-\frac{n}{m}\ell_{i}}=2^{-\frac{n}{m}}\rho(u_i),
\]
hence $\rho$ is $u$-regular.
Since the divergence and convergence of the sums
\eq{
\sum_{i=1}^{\infty} \frac{\psi(u_i)^s}{\rho(u_i)^m}\quad\text{and}\quad \sum_{i=1}^{\infty} \frac{\psi(2^{\ell_i})^s}{\rho(2^{\ell_i})^m}\quad \text{coincide},
} Theorem \ref{UbThm1} and Theorem \ref{LUb} imply Theorem \ref{thm_main}.
\end{proof}

\section{Proof of Theorem \ref{thm_counter}}\label{sec5}
In order to prove Theorem \ref{thm_counter}, we basically follow the proof of \cite[Theorem]{BL05}.

As in the introduction, let $(\mb{y}_k)_{k\geq 1}$ be a sequence of best approximations for $^t A$. Let $Y_k = \|\mb{y}_k\|$, $M_k = \|{^t A}\mb{y}_k\|_\bZ$, and
\[
\gamma_k = \max\left(\left(Y_{k}^{\frac{m}{n}}M_{k-1}\right)^{\frac{n}{m+n}}, \left(Y_{k+1}^{\frac{m}{n}}M_{k}\right)^{\frac{n}{m+n}}\right)
\]
for each $k \geq 2$.
For any $\al>0$, consider the set 
\[
B_\al(\{\gamma_k\}) = \{\mb{b}\in[0,1]^m : \|\mb{b}\cdot\mb{y}_k\|_\bZ > \al \ga_k \text{ for all large enough }k\geq 2\}.
\]
\begin{prop}\label{Prop_subset} For any $\al>0$,
\[B_\al(\{\gamma_k\})\subset \Bad_A\left(\frac{\al-n}{m}\right).\]
\end{prop}
\begin{proof}
Consider the following two sequences:
\[
U_k =\left( \frac{Y_k}{\gamma_k}\right)^{\frac{m}{n}}\quad\text{and}\quad V_k =\frac{\gamma_k}{M_k}.
\] 
We first claim that 
\begin{enumerate*}
\item\label{UVprop1} $V_k \to \infty$ as $k\to\infty$;
\item\label{UVprop2} $U_k < V_k$;
\item\label{UVprop3} $U_{k+1}\leq V_{k}$.
\end{enumerate*}
Indeed, \eqref{UVprop1} is clear, \eqref{UVprop2} follows from $M_k < M_{k-1}$ and $Y_k < Y_{k+1}$, and \eqref{UVprop3} follows from
\[
U_{k+1}\leq \left( \frac{Y_{k+1}}{(Y_{k+1}^{\frac{m}{n}}M_{k})^{\frac{n}{m+n}}}\right)^{\frac{m}{n}}=\frac{(Y_{k+1}^{\frac{m}{n}}M_k)^{\frac{n}{m+n}}}{M_k} \leq V_k.
\]
Hence the union $\bigcup_{k\geq2}[U_k, V_k)$ covers all sufficiently large numbers.

Now fix any $\mb{b}\in B_\al(\{\gamma_k\})$ and let $\mb{q}\in\bZ^n$ be an integral vector with sufficiently large norm.
Then we can find an index $k\geq 2$ with 
\eqlabel{Eq_UV}{
U_k \leq \|\mb{q}\| <V_k.
} Using the inequality in \cite[Equation (16)]{BL05}, we have
\[
\al \gamma_k < \|\mb{b}\cdot \mb{y}_k\|_\bZ \leq m Y_k \|A\mb{q}-\mb{b}\|_\bZ + n \|\mb{q}\| M_k, 
\] hence it follows from \eqref{Eq_UV} that 
\[
\|\mb{q}\|^{\frac{n}{m}}\|A\mb{q}-\mb{b}\|_\bZ > \frac{\al\gamma_k - n V_k M_k}{m Y_k} U_k^{\frac{n}{m}}=\frac{\al-n}{m}.
\] This proves the claim.
\end{proof}
\begin{proof}[Proof of Theorem \ref{thm_counter}]
It follows from the assumption $\sum_{k\geq 2} \ga_k <\infty$ and Borel-Cantelli lemma that $|B_\al(\{\gamma_k\})| =1$ for any $\al>0$. Given any $\del>0$, Proposition \ref{Prop_subset} with $\al=m\del+n$ implies Theorem \ref{thm_counter}.
\end{proof}

\def\cprime{$'$} \def\cprime{$'$} \def\cprime{$'$}
\providecommand{\bysame}{\leavevmode\hbox to3em{\hrulefill}\thinspace}
\providecommand{\MR}{\relax\ifhmode\unskip\space\fi MR }
\providecommand{\MRhref}[2]{%
  \href{http://www.ams.org/mathscinet-getitem?mr=#1}{#2}
}
\providecommand{\href}[2]{#2}

\end{document}